\documentclass[11pt]{article}

\usepackage[hidelinks]{hyperref}
\hypersetup{
  colorlinks   = true, 
  urlcolor     = blue, 
  linkcolor    = blue, 
  citecolor   = red 
}
\usepackage{amsmath,amsthm,amssymb}
\usepackage{graphicx}
\usepackage{bbm}
\newcommand{\sy}{\boldsymbol{\Psi}}
\newcommand{\py}{\boldsymbol{\Phi}}
\newcommand{\T}{\mathbb{T}^N}

\usepackage{xcolor}
\usepackage[margin=1in]{geometry} 
\usepackage{enumerate,mathtools,mathrsfs,bm,graphicx}
\usepackage[numbers, super]{natbib}
\usepackage[hidelinks]{hyperref}
\newcommand{\N}{\mathbb{N}}									

\newcommand{\R}{\mathbb{R}}

\newcommand{\vertiii}[1]{{\left\vert\kern-0.25ex\left\vert\kern-0.25ex\left\vert #1 
    \right\vert\kern-0.25ex\right\vert\kern-0.25ex\right\vert}}
    
\newcommand{\overbar}[1]{\mkern 1.5mu\overline{\mkern-1.5mu#1\mkern-1.5mu}\mkern 1.5mu}
\newcommand{\inner}[2]{\langle #1, #2 \rangle}
\DeclarePairedDelimiter\norm{\lVert}{\rVert}				

\newtheorem{theorem}{Theorem}[section]

\newtheorem{lemma}[theorem]{Lemma}
\newtheorem{proposition}[theorem]{Proposition}
\newtheorem*{remark}{Remark}
\newtheorem{definition}[theorem]{Definition}

\newtheorem{assumption}[theorem]{Assumption}

\begin{document}
	\title{Improved Blow-Up Criterion in a Variational Framework for Nonlinear SPDEs}
	\author{Daniel Goodair\footnote{EPFL, 
    daniel.goodair@epfl.ch. Supported
    by the EPSRC Project 2478902.}}
	\date{\today} 

\setcitestyle{numbers}	

\maketitle

\begin{abstract}
    We extend recent existence and uniqueness results for maximal solutions of SPDEs through an improved blow-up criterion. Whilst the maximal time of existence is typically characterised by blow-up in the energy norm of solutions, we show instead that solutions exist until blow-up in the larger spaces of the variational framework. The result is applied to show that solutions of 2D and 3D Stochastic Navier-Stokes Equations retain the higher order regularity of the initial condition on their time of existence. 
\end{abstract}



\section{Introduction}

The variational approach to nonlinear SPDEs with additive and multiplicative noise has long been studied, initially in the works [\cite{gyongy1980stochastic}, \cite{krylov2007stochastic}, \cite{pardoux1975equations}] and more recently [\cite{debussche2011local}, \cite{liu2010spde}, \cite{liu2013local}, \cite{neelima2020coercivity}] to name just a few contributions. Motivated by the physical relevance and potential regularising properties of \textit{transport noise}, where the stochastic integral depends on the gradient of the solution, there has been a trend towards unbounded noise in the variational framework as in [\cite{agresti2022nonlinear}, \cite{agresti2024critical}, \cite{alonso2021local}, \cite{goodair2024weak}, \cite{goodair2023existence},  \cite{rockner2022well}]. We similarly allow for an unbounded noise, which does not need to be small relative to coercivity.\\

Our result comes as a strict extension of the author's work [\cite{goodair2023existence}] with Crisan and Lang. There the authors proved the existence and uniqueness of maximal solutions with maximal time characterised by the blow-up in the energy norm of solutions. Under the exact same assumptions, here we prove that the maximal time can in fact be characterised by blow-up in a weaker norm given by the larger spaces of the framework. The method relies on Proposition \ref{amazing cauchy lemma}, recently proven by the author as a development of [\cite{glatt2009strong}] Lemma 5.1. The proposition is used to deduce the existence of a limiting process and stopping time under Cauchy and weak equicontinuity properties, as in [\cite{glatt2009strong}], with the novelty being that one can characterise the limit stopping time. In our application the stopping time is understood in terms of the first hitting time in the weaker norm, immediately implying the existence of solutions until blow-up in this norm.\\

Of the referenced literature we draw particular attention to [\cite{agresti2022nonlinear}, \cite{alonso2021local}]. In the latter the authors similarly prove the existence and uniqueness of maximal solutions until blow-up in the largest of the considered spaces. Their framework, however, does not ask for coercivity and in consequence solutions exhibit only pathwise continuity and not the additional square integrability. In the commonplace application of fluid dynamics, this setup is designed for \textit{inviscid} equations whereas ours is for \textit{viscous} equations. The applications and methodology are thus completely different, and we see these results as complementary. The former reference of [\cite{agresti2022nonlinear}] is by far the most comprehensive treatment of blow-up criteria in nonlinear SPDEs, where the notion of criticality is at the core of their work and spaces can be selected much more finely. In the vastness of their theory it is not entirely clear to what extent [\cite{agresti2022nonlinear}] could cover our results, though in any case we find value in our work through its simplicity and novel methodology.\\

We showcase the main result by proving that (local) strong solutions of 2D and 3D Stochastic Navier-Stokes Equations retain the higher order regularity of the initial condition on their time of existence. The argument is a simple iterated application of the variational result, as for $k \geq 2$ solutions are shown to blow-up in $C\left([0,T];W^{k,2}\right) \cap L^2\left([0,T];W^{k+1,2}\right)$ only if they blow-up in $C\left([0,T];W^{k-1,2}\right) \cap L^2\left([0,T];W^{k,2}\right)$ and inductively in $C\left([0,T];W^{1,2}\right) \cap L^2\left([0,T];W^{2,2}\right).$ To succinctly verify the assumptions we consider only a Lipschitz noise, though more exotic structures such as transport noise can certainly be considered.\\

The structure of the paper is as follows. We conclude this section with some brief stochastic preliminaries. In Section \ref{appendix local theory} we provide the setup, definitions and main result. Section \ref{proof sect} is devoted to the proof of the main result. The application to high order regularity of Stochastic Navier-Stokes is given in Section \ref{sect applications}. The key Proposition \ref{amazing cauchy lemma} is given as Supplementary Material which concludes the paper.

\subsection{Stochastic Preliminaries}

Let $(\Omega,\mathcal{F},(\mathcal{F}_t), \mathbb{P})$ be a fixed filtered probability space satisfying the usual conditions of completeness and right continuity. We take $\mathcal{W}$ to be a Cylindrical Brownian Motion over some Hilbert Space $\mathfrak{U}$ with orthonormal basis $(e_i)$. Given a process $F:[0,T] \times \Omega \rightarrow \mathscr{L}^2(\mathfrak{U};\mathcal{H})$ for $\mathscr{L}^2(\mathfrak{U};\mathcal{H})$ the Hilbert-Schmidt space, progressively measurable and such that $F \in L^2\left(\Omega \times [0,T];\mathscr{L}^2(\mathfrak{U};\mathcal{H})\right)$, for any $0 \leq t \leq T$ we define the stochastic integral $$\int_0^tF_sd\mathcal{W}_s:=\sum_{i=1}^\infty \int_0^tF_s(e_i)dW^i_s,$$ where the infinite sum is taken in $L^2(\Omega;\mathcal{H})$. We can extend this notion to processes $F$ which are such that $F(\omega) \in L^2\left( [0,T];\mathscr{L}^2(\mathfrak{U};\mathcal{H})\right)$ for $\mathbb{P}-a.e.$ $\omega$ via the traditional localisation procedure. A complete, direct construction of this integral, a treatment of its properties and the fundamentals of stochastic calculus in infinite dimensions can be found in [\cite{goodair2022stochastic}] Section 1.

\section{Setup and Main Result}
\label{appendix local theory}

\subsection{Functional Framework} \label{sub functional for local}

Our object of study is the It\^{o} SPDE  \begin{equation} \label{thespde}
    \sy_t = \sy_0 + \int_0^t \mathcal{A}(s,\sy_s)ds + \int_0^t\mathcal{G} (s,\sy_s) d\mathcal{W}_s
\end{equation}
which we pose for a triplet of embedded, separable Hilbert Spaces $$V \hookrightarrow H \hookrightarrow U$$ whereby the embeddings are continuous linear injections. We ask that there is a continuous bilinear form $\inner{\cdot}{\cdot}_{U \times V}: U \times V \rightarrow \R$ such that for $f \in H$ and $\psi \in V$, \begin{equation}  \nonumber
    \inner{f}{\psi}_{U \times V} =  \inner{f}{\psi}_{H}.
\end{equation}
The equation (\ref{thespde}) is posed on a time interval $[0,T]$ for arbitrary $T \geq 0$. The mappings $\mathcal{A},\mathcal{G}$ are such that
    $\mathcal{A}:[0,T] \times V \rightarrow U,
    \mathcal{G}:[0,T] \times V \rightarrow \mathscr{L}^2(\mathfrak{U};H)$ are measurable. Understanding $\mathcal{G}$ as a mapping $\mathcal{G}: [0,T] \times V \times \mathfrak{U} \rightarrow H$, we introduce the notation $\mathcal{G}_i(\cdot,\cdot):= \mathcal{G}(\cdot,\cdot,e_i)$. We further impose the existence of a system of elements $(a_k)$ of $V$ with the following properties. Let us define the spaces $V_n:= \textnormal{span}\left\{a_1, \dots, a_n \right\}$ and $\mathcal{P}_n$ as the orthogonal projection to $V_n$ in $U$. 
    It is required that the $(\mathcal{P}_n)$ are uniformly bounded in $H$, which is to say that there exists a constant $c$ independent of $n$ such that for all $\phi \in H$, \begin{equation}  \nonumber
        \norm{\mathcal{P}_nf}_H \leq c\norm{f}_H.
    \end{equation}
We also suppose that there exists a real valued sequence $(\mu_n)$ with $\mu_n \rightarrow \infty$ such that for any $f \in H$, \begin{align}
      \nonumber
    \norm{(I - \mathcal{P}_n)f}_U \leq \frac{1}{\mu_n}\norm{f}_{H}
\end{align}
where $I$ represents the identity operator in $U$. Specific bounds on the mappings $\mathcal{A}$ and $\mathcal{G}$ will be imposed in the following subsection. In order to make the assumptions we introduce some more notation here: we shall let $c_{\cdot}:[0,T]\rightarrow \R$ denote any bounded function, and for any constant $p \in \R$ we define the functions $K_U: U \rightarrow \R$, $K_H: H \rightarrow \R$, $K_V: V \rightarrow \R$ by
\begin{equation} \nonumber
    K_U(\phi)= 1 + \norm{\phi}_U^p, \quad K_H(\phi)= 1 + \norm{\phi}_H^p, \quad K_V(\phi)= 1 + \norm{\phi}_V^p.
\end{equation}
We may also consider these mappings as functions of two variables, e.g. $K_U: U \times U \rightarrow \R$ by $$K_U(\phi,\psi) = 1 + \norm{\phi}_U^p + \norm{\psi}_U^p.$$
Our assumptions will be stated for `the existence of a $K$ such that...' where we really mean `the existence of a $p$ such that, for the corresponding $K$, ...'.

\subsection{Assumptions} \label{assumptionschapter}

 We assume that there exists a $c_{\cdot}$, $K$ and $\gamma > 0$ such that for all $\phi,\psi \in V$, $\phi^n \in V_n$, $f \in H$ and $t \in [0,T]$:
 

  \begin{assumption}   \begin{align}
      \nonumber \norm{\mathcal{A}(t,\phi)}^2_U +\sum_{i=1}^\infty \norm{\mathcal{G}_i(t,\phi)}^2_H &\leq c_t K_U(\phi)\left[1 + \norm{\phi}_V^2\right],\\  \nonumber
     \norm{\mathcal{A}(t,\phi) - \mathcal{A}(t,\psi)}_U^2 &\leq  c_tK_V(\phi,\psi)\norm{\phi-\psi}_V^2,\\  \nonumber
    \sum_{i=1}^\infty \norm{\mathcal{G}_i(t,\phi) - \mathcal{G}_i(t,\psi)}_U^2 &\leq c_tK_U(\phi,\psi)\norm{\phi-\psi}_H^2.
 \end{align}
 \end{assumption}

\begin{assumption} \label{uniform assumpt}
 \begin{align}
    \nonumber  2\inner{\mathcal{P}_n\mathcal{A}(t,\phi^n)}{\phi^n}_H + \sum_{i=1}^\infty\norm{\mathcal{P}_n\mathcal{G}_i(t,\phi^n)}_H^2 &\leq c_tK_U(\phi^n)\left[1 + \norm{\phi^n}_H^4\right] - \gamma\norm{\phi^n}_V^2,\\  \nonumber
    \sum_{i=1}^\infty \inner{\mathcal{P}_n\mathcal{G}_i(t,\phi^n)}{\phi^n}^2_H &\leq c_tK_U(\phi^n)\left[1 + \norm{\phi^n}_H^6\right].
\end{align}
\end{assumption}

\begin{assumption} 
\begin{align}
  \nonumber 2\inner{\mathcal{A}(t,\phi) - \mathcal{A}(t,\psi)}{\phi - \psi}_U &+ \sum_{i=1}^\infty\norm{\mathcal{G}_i(t,\phi) - \mathcal{G}_i(t,\psi)}_U^2\\  \nonumber &\leq  c_{t}K_U(\phi,\psi)\left[1 + \norm{\phi}_H^2 + \norm{\psi}_H^2\right] \norm{\phi-\psi}_U^2 - \gamma\norm{\phi-\psi}_H^2,\\  \nonumber
    \sum_{i=1}^\infty \inner{\mathcal{G}_i(t,\phi) - \mathcal{G}_i(t,\psi)}{\phi-\psi}^2_U & \leq c_{t} K_U(\phi,\psi)\left[1 + \norm{\phi}_H^2 + \norm{\psi}_H^2\right] \norm{\phi-\psi}_U^4.
\end{align}
\end{assumption}

\begin{assumption} \label{my2.4}
\begin{align}
    \nonumber 2\inner{\mathcal{A}(t,\phi)}{\phi}_U + \sum_{i=1}^\infty\norm{\mathcal{G}_i(t,\phi)}_U^2 &\leq c_tK_U(\phi)\left[1 +  \norm{\phi}_H^2\right],\\ \nonumber
    \sum_{i=1}^\infty \inner{\mathcal{G}_i(t,\phi)}{\phi}^2_U &\leq c_tK_U(\phi)\left[1 + \norm{\phi}_H^4\right].
\end{align}
\end{assumption}

\begin{assumption} 
 \begin{equation}  \nonumber
    \inner{\mathcal{A}(t,\phi)-\mathcal{A}(t,\psi)}{f}_U \leq c_tK_U(\phi,\psi)(1+\norm{f}_H)\left[1 + \norm{\phi}_V + \norm{\psi}_V\right]\norm{\phi-\psi}_H.
    \end{equation}
\end{assumption}

\subsection{Definitions and Main Result} \label{subsection:notionsofsolution}

We state the definitions and main result.

\begin{definition} \label{v valued local def}
Let $\sy_0:\Omega \rightarrow H$ be $\mathcal{F}_0-$ measurable. A pair $(\sy,\tau)$ where $\tau$ is a $\mathbbm{P}-a.s.$ positive stopping time and $\sy$ is a process such that for $\mathbbm{P}-a.e.$ $\omega$, $\sy_{\cdot}(\omega) \in C\left([0,T];H\right)$ and $\sy_{\cdot}(\omega)\mathbbm{1}_{\cdot \leq \tau(\omega)} \in L^2\left([0,T];V\right)$ for all $T \geq 0$ and with $\sy_{\cdot}\mathbbm{1}_{\cdot \leq \tau}$ progressively measurable in $V$, is said to be a local strong solution of the equation (\ref{thespde}) if the identity
\begin{equation} \nonumber
    \sy_{t} = \sy_0 + \int_0^{t\wedge \tau} \mathcal{A}(s,\sy_s)ds + \int_0^{t \wedge \tau}\mathcal{G} (s,\sy_s) d\mathcal{W}_s
\end{equation}
holds $\mathbbm{P}-a.s.$ in $U$ for all $t \geq 0$.
\end{definition}

\begin{definition} \label{V valued maximal definition}
A pair $(\sy,\Theta)$ such that there exists a sequence of stopping times $(\theta_j)$ which are $\mathbbm{P}-a.s.$ monotone increasing and convergent to $\Theta$, whereby $(\sy_{\cdot \wedge \theta_j},\theta_j)$ is a local strong solution of the equation (\ref{thespde}) for each $j$, is said to be a maximal strong solution of the equation (\ref{thespde}) if for any other pair $(\py,\Gamma)$ with this property then $\Theta \leq \Gamma$ $\mathbbm{P}-a.s.$ implies $\Theta = \Gamma$ $\mathbbm{P}-a.s.$.
\end{definition}

\begin{remark}
We do not require $\Theta$ to be finite in this definition, in which case we mean that the sequence $(\theta_j)$ is monotone increasing and unbounded for such $\omega$. 
\end{remark}

\begin{definition} \label{v valued maximal unique}
A maximal strong solution $(\sy,\Theta)$ of the equation (\ref{thespde}) is said to be unique if for any other such solution $(\py,\Gamma)$, then $\Theta = \Gamma$ $\mathbbm{P}-a.s.$ and \begin{equation} \nonumber\mathbbm{P}\left(\left\{\omega \in \Omega: \sy_{t}(\omega) =  \py_{t}(\omega)  \quad \forall t \in [0,\Theta) \right\} \right) = 1. \end{equation}
\end{definition}

\begin{theorem} \label{theorem1}
For any given $\mathcal{F}_0-$ measurable $\sy_0:\Omega \rightarrow H$, there exists a unique maximal strong solution $(\sy,\Theta)$ of the equation (\ref{thespde}). Moreover at $\mathbbm{P}-a.e.$ $\omega$ for which $\Theta(\omega)<\infty$, we have that \begin{equation} \nonumber \sup_{r \in [0,\Theta(\omega))}\norm{\sy_r(\omega)}_U^2 + \int_0^{\Theta(\omega)}\norm{\sy_r(\omega)}_H^2dr = \infty\end{equation}
and in consequence for any $\mathbbm{P}-a.s.$ positive stopping time $\tau$ such that $$\sup_{r \in [0,\tau(\omega))}\norm{\sy_r(\omega)}_U^2 + \int_0^{\tau(\omega)}\norm{\sy_r(\omega)}_H^2dr < \infty $$
$\mathbbm{P}-a.s.$, $(\sy_{\cdot \wedge \tau}, \tau)$ is a local strong solution of the equation (\ref{thespde}).
\end{theorem}

\section{Proof of Theorem \ref{theorem1}} \label{proof sect}

This section is devoted to the proof of the main result, Theorem \ref{theorem1}. We recall that the assumptions are identical to those of [\cite{goodair2023existence}] Subsection 3.1. As an extension of [\cite{goodair2023existence}] Theorem 3.15, our first goal of this section is to summarise the method used in [\cite{goodair2023existence}]; this is the content of Subsection \ref{le synops}. Subsection \ref{le improve} then details how our new machinery of Proposition \ref{amazing cauchy lemma} facilitates the improved result of Theorem \ref{theorem1}, concluding its proof.

\subsection{A Synopsis of Our Approach} \label{le synops}

We first consider a bounded initial condition $\sy_0 \in L^{\infty}(\Omega;H)$ and the Galerkin Equations
\begin{equation} \label{nthorderGalerkin}
       \sy^n_t = \sy^n_0 + \int_0^t \mathcal{P}_n\mathcal{A}(s,\sy^n_s)ds + \int_0^t\mathcal{P}_n\mathcal{G} (s,\sy^n_s) d\mathcal{W}_s
\end{equation}
for $\sy^n_0:= \mathcal{P}_n\sy_0$ and $\mathcal{P}_n\mathcal{G} (e_i,s,\cdot):=\mathcal{P}_n\mathcal{G}_{i}(s,\cdot).$ Central to this work are two norms, for functions $\py \in L^\infty([0,T];U) \cap L^2([0,T];H)$, $\sy \in L^\infty([0,T];H) \cap L^2([0,T];V)$ defined by \begin{align} \nonumber 
    \norm{\py}^2_{UH,T}:&= \sup_{r \in [0,T]}\norm{\py_{r}}^2_U + \int_0^T\norm{\py_{r}}^2_Hdr\\
     \norm{\sy}^2_{HV,T}:&= \sup_{r \in [0,T]}\norm{\sy_{r}}^2_H + \int_0^T\norm{\sy_{r}}^2_Vdr.
     \nonumber 
\end{align}
The $HV$ norm corresponds to the regularity of strong solutions, so our idea is to show uniform regularity of the Galerkin Solutions in the $HV$ norm up until first hitting times in the lower $UH$ norm which sufficiently curbs the nonlinearity. These stopping times are defined for any $M > 1$ and $t \geq 0$ by 
\begin{equation}\label{tauMtn}\tau^{M,t}_n := t \wedge \inf\left\{s \geq 0: \norm{\sy^{n}}^2_{UH,s}  \geq M + \norm{\sy^n_0(\omega)}_U^2 \right\}.\end{equation}
For any such choices there exists a local strong solution $(\sy^n, \tau^{M,t}_n)$ of the equation (\ref{nthorderGalerkin}), see [\cite{goodair2023existence}] Lemma 3.18. Relying on Assumption \ref{uniform assumpt} then the uniform boundedness is proven, Proposition 3.21, stated here.

\begin{proposition} \label{theorem:uniformbounds}
There exists a constant $C$ dependent on $M,t$ but independent of $n$ such that for the local strong solution $(\sy^n, \tau^{M,t}_n)$ of (\ref{nthorderGalerkin}), \begin{equation} \label{firstresultofuniformbounds}
    \mathbbm{E}\norm{\sy^n}_{HV,\tau^{M,t}_{n}}^2\leq C\left[\mathbbm{E}\left(\norm{\sy^n_{0}}_H^2\right) + 1\right].
\end{equation}

\end{proposition}

We then look to use the result of Glatt-Holtz and Ziane, [\cite{glatt2009strong}] Lemma 5.1, to obtain a limiting process and positive stopping time as a candidate local strong solution of (\ref{thespde}). It is our extension of this result to Proposition \ref{amazing cauchy lemma} that is pivotal in the improved Theorem \ref{theorem1}, shown in the next subsection. To apply the Glatt-Holtz and Ziane result, the following were proven as Propositions 3.24 and 3.25.

\begin{proposition} \label{theorem:cauchygalerkin}
We have that
\begin{equation} \label{cauchy result pt 2}\lim_{m \rightarrow \infty}\sup_{n \geq m}\left[\mathbbm{E}\norm{\sy^n - \sy^m}_{UH,\tau^{M,t}_m \wedge \tau^{M,t}_n}^2\right] = 0.\end{equation}
\end{proposition}

\begin{proposition} \label{theorem: probability one}
We have that 
\begin{equation} \label{sufficient thing in probability condition}
    \lim_{S \rightarrow 0}\sup_{n \in \mathbb{N}}\mathbbm{E}\left[\norm{\sy^{n}}^2_{UH,\tau^{M,t}_n \wedge S} - \norm{\sy^n_0}_{U}^2 \right] = 0.
\end{equation}
\end{proposition}

This allows us to apply the Glatt-Holtz and Ziane result, obtaining Theorem 3.26 of [\cite{goodair2023existence}], stated here.

\begin{proposition} \label{existence of limiting pair theorem}
There exists a stopping time $\tau^{M,t}_{\infty}$, a subsequence $(\sy^{n_l})$ and a process $\sy_\cdot= \sy_{\cdot \wedge \tau^{M,t}_
\infty}$ whereby $\sy_{\cdot}\mathbbm{1}_{\cdot \leq \tau^{M,t}_{\infty}}$ is progressively measurable in $V$ and such that:
\begin{itemize}
    \item $\mathbb{P}\left(\left\{ 0 < \tau^{M,t}_{\infty} \leq \tau^{M,t}_{n_l}\right)\right\} = 1$;
    \item For $\mathbb{P}-a.e.$ $\omega$, $\sy^{n_l}(\omega) \rightarrow \sy(\omega)$ in $ L^\infty\left([0,\tau^{M,t}_{\infty}(\omega)];U\right) \cap L^2\left([0,\tau^{M,t}_{\infty}(\omega)];H\right)$, i.e. 
\begin{equation} \label{as convergence} \norm{\sy^{n_l}(\omega) - \sy(\omega)}_{UH,\tau^{M,t}_{\infty}(\omega)}^2 \longrightarrow 0;\end{equation}
\end{itemize}
\end{proposition}

From this point it is reasonably straightforwards to show that $(\sy, \tau^{M,t}_{\infty})$ is a local strong solution of (\ref{thespde}), as the uniform boundedness of Proposition \ref{theorem:uniformbounds} holds for the subsequence on $[0,\tau^{M,t}_{\infty}]$ allowing $\sy$ to inherit this regularity. The result is proven in Theorem 3.28, followed by the uniqueness, maximality and characterisation of the maximal time. Finally one can relieve the boundedness constraint on $\sy_0$ by partitioning $\Omega$ into sets on which an unbounded $\sy_0$ is bounded, using the unique maximal solution on each set, and piecing these together to obtain a solution for the unbounded $\sy_0$. This is done in Subsection 3.7.

\subsection{The Improved Method} \label{le improve}

Characterisation of the blow-up time in the previous subsection arises only through standard machinery on the energy norm of the solution, which is why the blow-up is given in the $HV$ norm. This machinery begins from the simple existence of \textit{a} local strong solution, making no use of the information that we have on $\tau^{M,t}_{\infty}$. It is clear, though, that $\tau^{M,t}_{\infty}$ is tightly connected with the $UH$ norm and the input parameters $M,t$. There is a strong intuition saying that for any first hitting time of $\sy$ in the $UH$ norm, we can choose $M$ and $t$ large enough so that $\tau^{M,t}_{\infty}$ exceeds it: such a property leads us to the fact that at the maximal time, which must be greater than all $\tau^{M,t}_{\infty}$, $\sy$ must blow-up in $UH$. Proposition \ref{amazing cauchy lemma} was developed to make this intuition rigorous. To apply it, we must upgrade the \textit{weak equicontinuity at time zero} from Proposition \ref{theorem: probability one} to a \textit{weak equicontinuity at all times}. 

\begin{lemma} \label{probably unnecessary lemma}
    Let $\theta$ be a stopping time and $(\delta_j)$ a sequence of stopping times which converge to $0$ $\mathbbm{P}-a.s.$. Then $$\lim_{j \rightarrow \infty}\sup_{n\in\N}\mathbbm{E}\left(\norm{\sy^n}_{UH,(\theta + \delta_j) \wedge \tau^{M,t}_n}^2 - \norm{\sy^n}_{UH,\theta \wedge \tau^{M,t}_n}^2 \right) =0.$$
\end{lemma}

\begin{proof}
       We look at the energy identity satisfied by $\sy^n$ up until the stopping time $\theta \wedge \tau^{M,T}_n$ and then $(\theta + r)  \wedge \tau^{M,T}_n$ for some $r \geq 0$. We have that
        \begin{align*}
        \norm{\sy^n_{\theta \wedge \tau^{M,t}_n}}_U^2 = \norm{\sy^n_0}_U^2 &+ 2\int_0^{\theta \wedge \tau^{M,t}_n}\inner{\mathcal{P}_n\mathcal{A}\left(s,\sy^n_s\right)}{\sy^n_s}_U ds  +\int_0^{\theta \wedge \tau^{M,t}_n}\sum_{i=1}^\infty \norm{\mathcal{P}_n\mathcal{G}_i\left(s,\sy^n_s\right)}_U^2ds\\ &+ 2\int_0^{\theta \wedge \tau^{M,t}_n}\inner{\mathcal{P}_n\mathcal{G}\left(s,\sy^n_s\right)}{\sy^n_s}_U d\mathcal{W}_s 
    \end{align*}
    and similarly for $(\theta + r)\wedge \tau^{M,t}_n$, from which the difference of the equalities gives
    \begin{align*}
        &\norm{\sy^n_{(\theta + r)\wedge \tau^{M,t}_n}}_U^2 = \norm{\sy^n_{\theta \wedge \tau^{M,t}_n}}_U^2 + 2\int_{\theta \wedge \tau^{M,t}_n}^{(\theta + r)\wedge \tau^{M,t}_n}\inner{\mathcal{P}_n\mathcal{A}\left(s,\sy^n_s\right)}{\sy^n_s}_U ds\\ &  \qquad \qquad  +\int_{\theta \wedge \tau^{M,t}_n}^{(\theta + r)\wedge \tau^{M,t}_n}\sum_{i=1}^\infty \norm{\mathcal{P}_n\mathcal{G}_i\left(s,\sy^n_s\right)}_U^2ds + 2\int_{\theta \wedge \tau^{M,t}_n}^{(\theta + r)\wedge \tau^{M,t}_n}\inner{\mathcal{P}_n\mathcal{G}\left(s,\sy^n_s\right)}{\sy^n_s}_U d\mathcal{W}_s.
    \end{align*}
Using that $\mathcal{P}_n$ is an orthogonal projection in $U$, and invoking Assumption \ref{my2.4}, we reduce to
     \begin{align*}
        \norm{\sy^n_{(\theta + r)\wedge \tau^{M,t}_n}}_U^2 &- \norm{\sy^n_{\theta \wedge \tau^{M,t}_n}}_U^2\\ & \leq c\int_{\theta \wedge \tau^{M,t}_n}^{(\theta + r)\wedge \tau^{M,t}_n}1 +  \norm{\sy^n_s}_H^2  ds + \int_{\theta \wedge \tau^{M,t}_n}^{(\theta + r)\wedge \tau^{M,t}_n}\inner{\mathcal{G}\left(s,\sy^n_s\right)}{\sy^n_s}_U d\mathcal{W}_s. 
    \end{align*}
where the constant $c$ depends on $M$, through a bound on the $U$ norm by the stopping time $\tau^{M,t}_n$.  We now take the supremum over $r\in[0,\delta_j]$ and expectation, using the Burkholder-Davis-Gundy Inequality,
\begin{align*}
        &\mathbbm{E}\left[\sup_{r \in [0,\delta_j]}\norm{\sy^n_{(\theta + r)\wedge \tau^{M,t}_n}}_U^2 - \norm{\sy^n_{\theta \wedge \tau^{M,t}_n}}_U^2 + \int_{\theta \wedge \tau^{M,t}_n}^{(\theta + \delta_j)  \wedge \tau^{M,t}_n} \norm{\sy^n_s}_H^2 ds\right]\\ & \qquad \leq  c\mathbbm{E}\int_{\theta \wedge \tau^{M,t}_n}^{(\theta + \delta_j)\wedge \tau^{M,t}_n}1 + \norm{\sy^n_s}_H^2ds + c\mathbbm{E}\left(\int_{\theta \wedge \tau^{M,t}_n}^{(\theta + \delta_j)\wedge \tau^{M,t}_n}\sum_{i=1}^\infty\inner{\mathcal{G}_i\left(s,\sy^n_s\right)}{\sy^n_s}^2_U ds\right)^{\frac{1}{2}}. 
    \end{align*}
    having then added an $\mathbbm{E}\int_{\theta \wedge \tau^{M,t}_n}^{(\theta + \delta_j)  \wedge \tau^{M,t}_n} \norm{\sy^n_s}_H^2 ds$ to both sides. Using the second part of Assumption \ref{my2.4}, again controlling the $U$ norm by a constant, we achieve that
\begin{align}
        \nonumber &\mathbbm{E}\left[\sup_{r \in [0,\delta_j]}\norm{\sy^n_{(\theta + r)\wedge \tau^{M,t}_n}}_U^2 - \norm{\sy^n_{\theta \wedge \tau^{M,t}_n}}_U^2 + \int_{\theta \wedge \tau^{M,t}_n}^{(\theta + \delta_j)  \wedge \tau^{M,t}_n} \norm{\sy^n_s}_H^2 ds\right]\\ & \qquad \qquad  \leq  c\mathbbm{E}\int_{\theta \wedge \tau^{M,t}_n}^{(\theta + \delta_j)\wedge \tau^{M,t}_n}1 + \norm{\sy^n_s}_H^2ds + c\mathbbm{E}\left(\int_{\theta \wedge \tau^{M,t}_n}^{(\theta + \delta_j)\wedge \tau^{M,t}_n}1 + \norm{\sy^n_s}_H^4 ds\right)^{\frac{1}{2}}. \label{itsnumbered}
    \end{align}
Attentions turn to the last term, for which we use Cauchy-Schwarz and Proposition \ref{theorem:uniformbounds} to obtain that
\begin{align*}
    \mathbbm{E}\left(\int_{\theta \wedge \tau^{M,t}_n}^{(\theta + \delta_j)\wedge \tau^{M,t}_n}1 + \norm{\sy^n_s}_H^4 ds\right)^{\frac{1}{2}} &\leq \mathbbm{E}\left(\sup_{r \in [0,\tau^{M,t}_n]}\norm{\sy^n_s}_H^2\int_{\theta \wedge \tau^{M,t}_n}^{(\theta + \delta_j)\wedge \tau^{M,t}_n}1 + \norm{\sy^n_s}_H^2 ds\right)^{\frac{1}{2}}\\
    &\leq \left[\mathbbm{E}\left(\sup_{r \in [0,\tau^{M,t}_n]}\norm{\sy^n_s}_H^2\right)\right]^{\frac{1}{2}}\left[\mathbbm{E}\int_{\theta \wedge \tau^{M,t}_n}^{(\theta + \delta_j)\wedge \tau^{M,t}_n}1 + \norm{\sy^n_s}_H^2 ds\right]^{\frac{1}{2}}\\
    &\leq c\left[\mathbbm{E}\int_{\theta \wedge \tau^{M,t}_n}^{(\theta + \delta_j)\wedge \tau^{M,t}_n}1 + \norm{\sy^n_s}_H^2 ds\right]^{\frac{1}{2}}
\end{align*}
where $c$ is dependent on the boundedness of the initial condition $\sy_0$ from Proposition \ref{theorem:uniformbounds}. For both this control and the remaining term of (\ref{itsnumbered}), we show that 
\begin{align*}
    &\mathbbm{E}\int_{\theta \wedge \tau^{M,t}_n}^{(\theta + \delta_j)\wedge \tau^{M,t}_n} \norm{\sy^n_s}_H^2 ds\\ & \qquad \qquad \qquad \qquad \leq \mathbbm{E}\left(\sup_{r\in[0, \tau^{M,t}_n]}\norm{\sy^n_r}_H\int_{\theta \wedge \tau^{M,t}_n}^{(\theta + \delta_j)\wedge \tau^{M,t}_n} \norm{\sy^n_s}_H ds\right)\\
    &\qquad \qquad \qquad \qquad \leq \left[\mathbbm{E}\left(\sup_{r\in[0, \tau^{M,t}_n]}\norm{\sy^n_r}_H^2\right)\right]^{\frac{1}{2}}\left[\mathbbm{E}\left(\int_{\theta \wedge \tau^{M,t}_n}^{(\theta + \delta_j)\wedge \tau^{M,t}_n} \norm{\sy^n_s}_H ds\right)^2\right]^{\frac{1}{2}}\\
     &\qquad \qquad \qquad \qquad \leq c\left[\mathbbm{E}\left(\int_{\theta \wedge \tau^{M,t}_n}^{(\theta + \delta_j)\wedge \tau^{M,t}_n} \norm{\sy^n_s}_H ds\right)^2\right]^{\frac{1}{2}}\\
    &\qquad \qquad \qquad \qquad \leq c\left[\mathbbm{E}\left(\left(\int_{\theta \wedge \tau^{M,t}_n}^{(\theta + \delta_j)\wedge \tau^{M,t}_n} 1 + \norm{\sy^n_s}_H^2 ds\right)\left(\int_{\theta \wedge \tau^{M,t}_n}^{(\theta + \delta_j)\wedge \tau^{M,t}_n} \norm{\sy^n_s}_H ds\right)\right)\right]^{\frac{1}{2}}\\
    &\qquad \qquad \qquad \qquad \leq c\left[\mathbbm{E}\left(\int_{\theta \wedge \tau^{M,t}_n}^{(\theta + \delta_j)\wedge \tau^{M,t}_n} \norm{\sy^n_s}_H ds\right)\right]^{\frac{1}{2}}\\
    &\qquad \qquad \qquad \qquad \leq c\left[\left[\mathbbm{E}\left(\sup_{r\in[0, \tau^{M,t}_n]}\norm{\sy^n_r}_H^2\right)\right]^{\frac{1}{2}}\left[\mathbbm{E}\left(\int_{\theta \wedge \tau^{M,t}_n}^{(\theta + \delta_j)\wedge \tau^{M,t}_n}1 ds\right)^2\right]^{\frac{1}{2}}\right]^{\frac{1}{2}}\\
    &\qquad \qquad \qquad \qquad \leq c\left[\mathbbm{E}(\delta_j^2)\right]^{\frac{1}{4}}
\end{align*}
having utilised that $\int_0^{\tau^{M,t}_n}\norm{\sy^n_s}_H^2ds \leq c$ again from the definition of the first hitting time. Noting that $\delta_j$ is $\mathbbm{P}-a.s.$ monotone decreasing (as $j \rightarrow \infty$) and  convergent to $0$, the Monotone Convergence Theorem thus justifies that $$c\left[\mathbbm{E}(\delta_j^2)\right]^{\frac{1}{4}} = o_j$$
where $o_j$ represents a constant independent of $n$ which goes to zero as $\delta_j \rightarrow 0$. Revisiting (\ref{itsnumbered}), we have now justified that
\begin{equation} \label{a la combo} \mathbbm{E}\left[\sup_{r \in [0,\delta_j]}\norm{\sy^n_{(\theta + r)\wedge \tau^{M,t}_n}}_U^2 - \norm{\sy^n_{\theta \wedge \tau^{M,t}_n}}_U^2 + \int_{\theta \wedge \tau^{M,t}_n}^{(\theta + \delta_j)  \wedge \tau^{M,t}_n} \norm{\sy^n_s}_H^2 ds\right] \leq o_j.\end{equation}
It remains to relate the expression on the left hand side with what we are interested in, which is $\norm{\sy^n}_{UH,(\theta + \delta_j) \wedge \tau^{M,t}_n}^2 - \norm{\sy^n}_{UH,\theta \wedge \tau^{M,t}_n}^2$. We have that \begin{align*}\norm{\sy^n}_{UH,(\theta + \delta_j) \wedge \tau^{M,t}_n}^2 &- \norm{\sy^n}_{UH,\theta \wedge \tau^{M,t}_n}^2\\ &= \sup_{s \in [0,(\theta + \delta_j)\wedge \tau^{M,t}_n]}\norm{\sy^n_s}_U^2 -\sup_{s \in [0,\theta \wedge \tau^{M,t}_n]}\norm{\sy^n_s}_U^2 + \int_{\theta \wedge \tau^{M,t}_n}^{(\theta + \delta_j)  \wedge \tau^{M,t}_n} \norm{\sy^n_s}_H^2 ds\end{align*} and claim \begin{equation}\label{theclaim}\sup_{s \in [0,(\theta + \delta_j)\wedge \tau^{M,t}_n]}\norm{\sy^n_s}_U^2 -\sup_{s \in [0,\theta \wedge \tau^{M,t}_n]}\norm{\sy^n_s}_U^2 \leq \sup_{r \in [0,\delta_j]}\norm{\sy^n_{(\theta + r)\wedge \tau^{M,t}_n}}_U^2 - \norm{\sy^n_{\theta \wedge \tau^{M,t}_n}}_U^2.\end{equation}
Indeed, we have that $$ \sup_{s \in [0,(\theta + \delta_j)\wedge \tau^{M,t}_n]}\norm{\sy^n_s}_U^2 \leq \sup_{s \in [0,\theta \wedge \tau^{M,t}_n]}\norm{\sy^n_s}_U^2 + \sup_{s \in [\theta \wedge \tau^{M,t}_n,(\theta + \delta_j)\wedge \tau^{M,t}_n]}\norm{\sy^n_s}_U^2 - \norm{\sy^n_{\theta \wedge \tau^{M,t}_n}}_U^2$$ as the left hand side must equal either $\sup_{s \in [0,\theta \wedge \tau^{M,t}_n]}\norm{\sy^n_s}_U^2$ or $\sup_{s \in [\theta \wedge \tau^{M,t}_n,(\theta + \delta_j)\wedge \tau^{M,t}_n]}\norm{\sy^n_s}_U^2$, both of which are greater than the subtracted term $\norm{\sy^n_{\theta \wedge \tau^{M,t}_n}}_U^2$. Appreciating that $$\sup_{s \in [\theta \wedge \tau^{M,t}_n,(\theta + \delta_j)\wedge \tau^{M,t}_n]}\norm{\sy^n_s}_U^2 = \sup_{r \in [0,\delta_j]}\norm{\sy^n_{(\theta + r)\wedge \tau^{M,t}_n}}_U^2 $$ then yields the claim (\ref{theclaim}), which in combination with (\ref{a la combo}) grants that
\begin{align} \nonumber
        \mathbbm{E}\left[\norm{\sy^n}_{UH,(\theta + \delta_j) \wedge \tau^{M,t}_n}^2 - \norm{\sy^n}_{UH,\theta \wedge \tau^{M,t}_n}^2\right] \leq o_j.
    \end{align}
This proves the result.
   
\end{proof}

In combination with Proposition \ref{theorem:cauchygalerkin} we are entitled to apply Proposition \ref{amazing cauchy lemma} for the spaces $X_s := L^\infty\left([0,s];U\right) \cap L^2\left([0,s];H\right)$ with $\norm{\cdot}_{UH,s}$ norm. We obtain, therefore, that for any $t \geq 0$ and any given $R>0$ we can choose $M$ such that the $\tau^{M,t}_{\infty}$ of Proposition \ref{existence of limiting pair theorem}, for which $(\sy, \tau^{M,t}_{\infty})$ is a local strong solution, satisfies $\tau^{R,t} \leq \tau^{M,t}_{\infty}$ $\mathbbm{P}-a.s.$ with
$$\tau^{R,t}:= t \wedge \inf\left\{s \geq 0: \norm{\sy}_{UH,s}^2 \geq R \right\}.$$
We now inspect how this affects the maximal time $\Theta$, and argue that at $\mathbbm{P}-a.e.$ $\omega$ for which $\Theta(\omega)<\infty$, \begin{equation} \label{we show it} \norm{\sy(\omega)}_{UH,\Theta(\omega)}^2 := \sup_{r \in [0,\Theta(\omega))}\norm{\sy_r(\omega)}_U^2 + \int_0^{\Theta(\omega)}\norm{\sy_r(\omega)}_H^2dr = \infty.\end{equation}
The maximality of $\Theta$ ensures that for every $R$ and $t$, $\tau^{R,t} \leq \Theta$ $\mathbbm{P}-a.s.$ as it must exceed any stopping time which is the lifetime of a local strong (see for example, [\cite{goodair2023existence}] Corollary 3.35). Suppose for a contradiction that there exists a set of positive probability on which $\Theta < \infty$ and $\norm{\sy}_{UH,\Theta}^2 < \infty$. Classically there must exist a set of positive probability $\mathscr{A}$ and values $t,R$ such that for all $\omega \in \mathscr{A}$, $\Theta < t$ and $\norm{\sy}_{UH,\Theta}^2 < R$. This implies that for $\omega \in \mathscr{A}$, $\Theta(\omega) < \tau^{R,t}(\omega)$ which provides the contradiction. The property (\ref{we show it}) is thus proven.\\

This proves the first assertion of Theorem \ref{theorem1} in the case that $\sy_0 \in L^\infty\left(\Omega;H\right)$. Extension to the unbounded case is exactly as in Subsection \ref{le synops}, c.f. [\cite{goodair2023existence}] Subsection 3.7, which preserves this blow-up at the maximal time. To completely prove Theorem \ref{theorem1} we now only need to justify the second assertion. To this end we take a positive stopping time $\tau$ such that $\norm{\sy}_{UH,\tau}^2 < \infty$ $\mathbbm{P}-a.s.$. We must show that 
$\sy_{\cdot \wedge \tau} \in C\left([0,T];H\right)$ and $\sy_{\cdot}\mathbbm{1}_{\cdot \leq \tau} \in L^2\left([0,T];V\right)$ for all $T \geq 0$ $\mathbbm{P}-a.s.$, that $\sy_{\cdot \wedge \tau}\mathbbm{1}_{\cdot \leq \tau}$ progressively measurable in $V$, and that the identity
\begin{equation} \nonumber
    \sy_{t \wedge \tau} = \sy_0 + \int_0^{t\wedge \tau} \mathcal{A}(s,\sy_s)ds + \int_0^{t \wedge \tau}\mathcal{G} (s,\sy_s) d\mathcal{W}_s
\end{equation}
holds $\mathbbm{P}-a.s.$ in $U$ for all $t \geq 0$. By definition of $\Theta$ there exists a sequence of stopping times $(\theta_j)$ $\mathbbm{P}-a.s.$ monotone increasing and convergent to $\Theta$, whereby $(\sy_{\cdot \wedge \theta_j},\theta_j)$ is a local strong solution of the equation (\ref{thespde}) for each $j$. We consider the possibility that $\tau$ is infinite on some measurable subset $\mathscr{B} \subset \Omega$, implying that $\Theta = \infty$ on $\mathscr{B}$ by the blow-up, and also note that on $\mathscr{B}^C$ where $\tau < \infty$ then again due to blow-up we must have that $\tau < \Theta$. In any case, $(\theta_j \wedge \tau)$ is monotone convergent to $\tau$ $\mathbbm{P}-a.s.$, and $(\sy_{\cdot \wedge \theta_j \wedge \tau}, \theta_j \wedge \tau)$ is a local strong solution of the equation (\ref{thespde}). For the progressive measurability, for each fixed $T > 0$ we understand $\sy_{\cdot \wedge \tau}\mathbbm{1}_{\cdot \leq \tau}$ as the $\mathbbm{P} \times \lambda-a.e.$ limit of the $\mathcal{F}_T \times \mathcal{B}([0,T])-$measurable $(\sy_{\cdot \wedge \theta_j \wedge \tau}\mathbbm{1}_{\cdot \leq \theta_j \wedge \tau})$ on $\Omega \times [0,T]$. Such a limit preserves the measurability on the product sigma algebra, justifying the required progressive measurability. The remaining properties are pathwise hence even clearer, as for $\mathbbm{P}-a.e.$ $\omega$ in $\mathscr{B}$ and for any given $T$ there exists a $j$ such that $\theta_j(\omega) > T$ and $\sy_{\cdot \wedge \theta_j}$ has the required regularity. Similarly on $\mathscr{B}^C$ for $\mathbbm{P}-a.e.$ $\omega$ there exists a $j$ such that $\tau(\omega) < \theta_j(\omega)$, so $\sy_{\cdot \wedge \tau(\omega)}(\omega) = \sy_{\cdot \wedge \theta_j(\omega) \wedge \tau(\omega)}(\omega)$ which has the necessary properties as $\sy_{\cdot \wedge \theta_j \wedge \tau}$ is a local strong solution. This concludes the proof.

\section{Application: High Order Regularity for Stochastic Navier-Stokes} \label{sect applications}

As an application of this improved blow-up criterion, we demonstrate high order regularity of a Stochastic Navier-Stokes Equation
\begin{equation} \label{number2equationSALT}
    u_t = u_0 - \int_0^t\mathcal{L}_{u_s}u_s\,ds + \nu\int_0^t \Delta u_s\, ds + \int_0^t \mathcal{G}(u_s) d\mathcal{W}_s - \nabla \rho_t
\end{equation}
where $u$ represents the fluid velocity, $\nu > 0$ the viscosity, $\rho$ the pressure and $\mathcal{L}$ the nonlinear term defined by $\mathcal{L}_fg = \sum_{j=1}^Nf^j\partial_jg$ with Laplacian $\Delta f = \sum_{j=1}^N \partial_j^2 f$. We pose the equation over the torus $\T$ in $N=2$ or $3$ dimensions. On the noise $\mathcal{G}$ we assume that for each $i, j \in \N$, there exists constants $c_{i,j}$ such that $\mathcal{G}_i$ is $c_{i,j}-$Lipschitz on $W^{k,2}(\T;\R^N)$ and $\sum_{i=1}^\infty c_{i,j}^2 < \infty$. Of course more exciting noise structures could be considered in this framework, but we choose the Lipschitz case for a simple demonstration. We require the divergence-free property of solutions, which is to say that $\sum_{j=1}^N\partial_ju^j = 0$. To facilitate the analysis, we introduce some additional function spaces. Recall that any function $f \in L^2(\T;\R^N)$ admits the representation \begin{equation} \label{fourier rep}f(x) = \sum_{k \in \mathbb{Z}^N}f_ke^{ik\cdot x}\end{equation} whereby each $f_k \in \mathbb{C}^N$ is such that $f_k = \overbar{f_{-k}}$ and the infinite sum is defined as a limit in $L^2(\T;\R^N)$, see e.g. [\cite{robinson2016three}] Subsection 1.5 for details.

\begin{definition}
We define $L^2_{\sigma}$ as the subset of $L^2(\T;\R^N)$ of zero-mean functions $f$ whereby for all $k \in \mathbbm{Z}^N$, $k \cdot f_k = 0$ with $f_k$ as in (\ref{fourier rep}). For general $m \in \N$ we introduce $W^{m,2}_{\sigma}$ as the intersection of $W^{m,2}(\T;\R^N)$ respectively with $L^2_{\sigma}$.
\end{definition}

Note that the dimensionality $N$ is not explicitly included in the spaces, but will be made clear from context. We define the Leray Projector $\mathcal{P}$ as the orthogonal projection in $L^2(\T;\R^N)$ onto $L^2_\sigma$. For $m \in \N$ the inner product $\inner{f}{g}_m:= \inner{(-\mathcal{P}\Delta)^{m/2}f}{(-\mathcal{P}\Delta)^{m/2}g}$ is equivalent to the usual $W^{m,2}(\T;\R^N)$ inner product on $W^{m,2}_{\sigma}$ and we consider $W^{m,2}_{\sigma}$ as a Hilbert Space equipped with this inner product. Further details can be found in [\cite{robinson2016three}] Exercises 2.12, 2.13 and the discussion in Subsection 2.3. Following the typical study of incompressible Navier-Stokes, we work with the projected equation 
\begin{equation} \label{projected Ito}
    u_t = u_0 - \int_0^t\mathcal{P}\mathcal{L}_{u_s}u_s\ ds + \nu\int_0^t \mathcal{P} \Delta u_s\, ds  + \int_0^t \mathcal{P}\mathcal{G}(u_s) d\mathcal{W}_s 
\end{equation}
which is now in the form of (\ref{thespde}). The existence of a unique local strong solution to (\ref{projected Ito}) in 3D, and a unique global strong solution in 2D, is by this point standard: see for instance, [\cite{glatt2009strong}, \cite{goodair2024weak}, \cite{goodair20233d}]. We state the result here.

\begin{proposition} \label{navier strong existence}
     Let $u_0: \Omega \rightarrow W^{1,2}_{\sigma}$ be $\mathcal{F}_0-$measurable. Then there exists a pair $(u,\tau)$ where $\tau$ is a $\mathbbm{P}-a.s.$ positive stopping time and $u$ is a process such that for $\mathbbm{P}-a.e.$ $\omega$, $u_{\cdot}(\omega) \in C\left([0,T];W^{1,2}_{\sigma}\right)$ and $u_{\cdot}(\omega)\mathbbm{1}_{\cdot \leq \tau(\omega)} \in L^2\left([0,T];W^{2,2}_{\sigma}\right)$ for all $T \geq 0$ and with $u_{\cdot}\mathbbm{1}_{\cdot \leq \tau}$ progressively measurable in $W^{2,2}_{\sigma}$, satisfying
     \begin{equation} \nonumber
    u_t = u_0 - \int_0^{t\wedge \tau}\mathcal{P}\mathcal{L}_{u_s}u_s\ ds + \nu\int_0^{t \wedge \tau} \mathcal{P} \Delta u_s\, ds  + \int_0^{t \wedge \tau} \mathcal{P}\mathcal{G}(u_s) d\mathcal{W}_s 
\end{equation}
$\mathbbm{P}-a.s.$ in $L^2_{\sigma}$ for all $t \geq 0$. Moreover if $(v, \gamma)$ was any other such local strong solution then  \begin{equation} \nonumber\mathbbm{P}\left(\left\{\omega \in \Omega: u_{t}(\omega) =  v_{t}(\omega)  \quad \forall t \in [0,\tau \wedge \gamma] \right\} \right) = 1. \end{equation}
If $N=2$ then for any given $T > 0$ one can choose $\tau := T$.
\end{proposition}

The problem that we consider is, if $u_0: \Omega \rightarrow W^{k,2}_{\sigma}$ is $\mathcal{F}_0-$measurable for some $k \in \N$, then does $u$ belong pathwise to $C\left([0,\tau];W^{k,2}_{\sigma}\right) \cap L^2\left([0,\tau];W^{k+1,2}_{\sigma}\right)$? The result is affirmative and proven through iterated applications of Theorem \ref{theorem1}. In fact the framework and assumptions of Subsections \ref{sub functional for local} and \ref{assumptionschapter} were completely verified for the spaces \begin{equation}\label{copmpletely verified}V:= W^{3,2}_{\sigma}, \qquad H:= W^{2,2}_{\sigma}, \qquad U:= W^{1,2}_{\sigma}\end{equation}
under a transport noise in [\cite{goodair20233d}] Section 3, whilst much more easily holding for the Lipschitz noise.  The maximal solution that we obtain must agree with $u$ on its lifetime of existence by uniqueness. From Proposition \ref{navier strong existence} it is certainly true that $\sup_{r \in [0,\tau)}\norm{u_r}_{W^{1,2}_{\sigma}}^2 + \int_0^{\tau}\norm{u_r}_{W^{2,2}_{\sigma}}^2dr < \infty $ $\mathbbm{P}-a.s.$, so from Theorem \ref{theorem1} with the spaces established in (\ref{copmpletely verified}) we verify that for $\mathbbm{P}-a.e.$ $\omega$, $u_{\cdot}(\omega) \in C\left([0,T];W^{2,2}_{\sigma}\right)$\footnote{Recall that $u_{\cdot} = u_{\cdot \wedge \tau}$.} and $u_{\cdot}(\omega)\mathbbm{1}_{\cdot \leq \tau(\omega)} \in L^2\left([0,T];W^{3,2}_{\sigma}\right)$ for all $T \geq 0$. In particular, $\sup_{r \in [0,\tau)}\norm{u_r}_{W^{2,2}_{\sigma}}^2 + \int_0^{\tau}\norm{u_r}_{W^{3,2}_{\sigma}}^2dr < \infty $. The inductive method is now apparent, where we consider spaces 
\begin{equation}\nonumber V:= W^{j+1,2}_{\sigma}, \qquad H:= W^{j,2}_{\sigma}, \qquad U:= W^{j-1,2}_{\sigma}.\end{equation}
For the Lipschitz noise a verification of the assumptions in these higher spaces provides little additional difficulty to the case of (\ref{copmpletely verified}), so we omit the complete details here and content ourselves with applying Theorem \ref{theorem1} in any such case. Repeating this procedure for $j=3, 4, \dots, k$, we show the following.


\begin{theorem}
    For any given $k \in \N$ let $u_0: \Omega \rightarrow W^{k,2}_{\sigma}$ be $\mathcal{F}_0-$measurable. Then any local strong solution $(u,\tau)$ of (\ref{projected Ito}) as specified in Proposition \ref{navier strong existence} is such that for $\mathbb{P}-a.e.$ $\omega$, $u_{\cdot}(\omega) \in C\left([0,T];W^{k,2}_{\sigma}\right)$ and $u_{\cdot}(\omega)\mathbbm{1}_{\cdot \leq \tau(\omega)} \in L^2\left([0,T];W^{k+1,2}_{\sigma}\right)$ for all $T \geq 0$. If $N = 2$ then for $\mathbb{P}-a.e.$ $\omega$, $u_{\cdot}(\omega) \in C\left([0,T];W^{k,2}_{\sigma}\right)\cap L^2\left([0,T];W^{k+1,2}_{\sigma}\right)$ for all $T \geq 0$.
\end{theorem}

It should be noted that the assumptions cannot be verified for the spaces $V:= W^{2,2}_{\sigma}$, $H:= W^{1,2}_{\sigma}$, $U:= L^2_{\sigma}$ as the algebra property for $H$ is lost. Furthermore we do not obtain strong solutions in 3D on the lifespan of weak solutions.


\section{Appendix}

\begin{proposition} \label{amazing cauchy lemma}
    Fix $T>0$. For $t\in[0,T]$ let $X_t$ denote a Banach Space with norm $\norm{\cdot}_{X,t}$ such that for all $s > t$, $X_s \xhookrightarrow{}X_t$ and $\norm{\cdot}_{X,t} \leq \norm{\cdot}_{X,s}$. Suppose that $(\sy^n)$ is a sequence of processes $\sy^n:\Omega \mapsto X_T$, $\norm{\sy^n}_{X,\cdot}$ is adapted and $\mathbbm{P}-a.s.$ continuous, $\sy^n \in L^2\left(\Omega;X_T\right)$, and such that $\sup_{n}\norm{\sy^n}_{X,0} \in L^\infty\left(\Omega;\R\right)$. For any given $M >1$ define the stopping times
    \begin{equation} \label{another taumt}
        \tau^{M,T}_n := T \wedge \inf\left\{s \geq 0: \norm{\sy^n}_{X,s}^2 \geq M + \norm{\sy^n}_{X,0}^2 \right\}.
    \end{equation}
Furthermore suppose \begin{equation} \label{supposition 1}
    \lim_{m \rightarrow \infty}\sup_{n \geq m}\mathbbm{E}\left[\norm{\sy^n-\sy^m}^2_{X,\tau
_{m}^{M,t}\wedge \tau_{n}^{M,t}} \right] =0
\end{equation}
and that for any stopping time $\gamma$ and sequence of stopping times $(\delta_j)$ which converge to $0$ $\mathbbm{P}-a.s.$, \begin{equation} \label{supposition 2} \lim_{j \rightarrow \infty}\sup_{n\in\N}\mathbbm{E}\left(\norm{\sy^n}_{X,(\gamma + \delta_j) \wedge \tau^{M,T}_n}^2 - \norm{\sy^n}_{X,\gamma \wedge \tau^{M,T}_n}^2 \right) =0.
\end{equation}
Then there exists a stopping time $\tau^{M,T}_{\infty}$, a process $\sy:\Omega \mapsto X_{\tau^{M,T}_{\infty}}$ whereby $\norm{\sy}_{X,\cdot \wedge \tau^{M,T}_{\infty}}$ is adapted and $\mathbbm{P}-a.s.$ continuous, and a subsequence indexed by $(m_j)$ such that 
\begin{itemize}
    \item $\tau^{M,T}_{\infty} \leq \tau^{M,T}_{m_j}$ $\mathbbm{P}-a.s.$,
    \item $\lim_{j \rightarrow \infty}\norm{\sy - \sy^{m_j}}_{X,\tau^{M,T}_{\infty}} = 0$ $\mathbbm{P}-a.s.$.
\end{itemize}
Moreover for any $R>0$ we can choose $M$ to be such that the stopping time \begin{equation} \label{another tauR}
        \tau^{R,T} := T \wedge \inf\left\{s \geq 0: \norm{\sy}_{X,s\wedge\tau^{M,T}_{\infty}}^2 \geq R \right\}
    \end{equation}
satisfies $\tau^{R,T} \leq \tau^{M,T}_{\infty}$ $\mathbbm{P}-a.s.$. Thus $\tau^{R,T}$ is simply $T \wedge \inf\left\{s \geq 0: \norm{\sy}_{X,s}^2 \geq R \right\}$.

\end{proposition}

\begin{proof}
    See [\cite{goodair2024weak}] Proposition 6.1.
\end{proof}




\bibliographystyle{newthing}
\bibliography{mybibo}

\end{document}